\documentclass[a4paper,11pt]{article}
\usepackage{amsmath,amssymb,amsthm}
\usepackage[cp1251]{inputenc}
\usepackage[russian, english]{babel}
\inputencoding{cp1251}

\usepackage{graphicx}

\usepackage[pdfpagemode=UseOutlines,
 bookmarks=true,bookmarksopen=true,bookmarksnumbered=true,
 pdffitwindow=true,pdfpagelayout=SinglePage,pdfhighlight=/P,
 colorlinks=true,linkcolor=blue,citecolor=blue,urlcolor=blue,
 unicode=true]{hyperref}

\theoremstyle{plain}

\newtheorem{lemma}{Lemma}
\numberwithin{lemma}{section}
\newtheorem{corollary}{Corollary}
\numberwithin{corollary}{section}
\newtheorem{theorem}{Theorem}
\numberwithin{theorem}{section}
\newtheorem{proposition}{Proposition}
\numberwithin{proposition}{section}

\theoremstyle{definition}
\newtheorem{definition}{Definition}
\numberwithin{definition}{section}
\newtheorem{remark}{Remark}
\numberwithin{remark}{section}
\newtheorem{example}{Example}
\numberwithin{example}{section}

\title{The $D_8$-tower of weak Jacobi forms and applications}
\author{Dmitry Adler and Valery Gritsenko
\footnote 
{The  authors are supported by the Laboratory of Mirror Symmetry NRU HSE 
(RF government grant, ag. no. 14.641.31.0001).}}
\date{\today}

\begin{document}\maketitle

%-------------------------------------------------------------------------------
%\thispagestyle{empty}
\begin{abstract}
We construct a tower of arithmetic  generators of the bigraded polynomial ring $J_{*,*}^{w, O}(D_n)$ of weak 
Jacobi modular forms invariant with respect to the full orthogonal group
$O(D_n)$ of the root lattice $D_n$ for $2\le n\le 8$. This tower 
corresponds to the tower of strongly reflective modular forms on the 
orthogonal groups of signature $(2,n)$ which determine the  Lorentzian 
Kac-Moody algebras related to the BCOV (Bershadsky-Cecotti-Ooguri-Vafa)-analytic 
torsions. We prove that the main three generators of index one of 
the graded ring satisfy a special system of modular differential 
equations. We found also a general modular differential equation of the 
generator of weight $0$ and index $1$ which generates the automorphic 
discriminant of the moduli space of Enriques surfaces.    
\end{abstract}

\section{Introduction}
In this paper we construct the tower of arithmetic  generators of the 
bigraded polynomial ring $J_{*,*}^{w, O}(D_n)$ of $O(D_n)$-invariant weak 
Jacobi modular forms with respect to the root lattice $D_n$ for 
$2\le n\le 8$, corresponding 
to the $D_8$-tower of strongly reflective modular forms on the orthogonal 
groups $O^+(2U\oplus D_n(-1))$ 
(see \cite{Gr18}). The strongly reflective modular forms 
of this tower of the orthogonal groups 
${\widetilde O^+(2U\oplus D_n(-1))}$ $(3\le n\le 8)$ determine the 
Lorentzian Kac--Moody algebras corresponding to the 
BCOV (Bershadsky-Cecotti-Ooguri-Vafa)-analytic torsions (see \cite{Y} and \cite{Gr18}). 

The fact that the bigraded ring  $J_{*,*}^{W}(D_n)$
of Jacobi forms invariants with respect to the Weyl group of $D_n$  
is polynomial was proved  
in \cite{Wirt}. The Wirtm\"uller theorem is the analog of the Chevalley 
theorem for invariant polynomials of a Coxeter group 
(see \cite{Lo1}--\cite{Lo2}, \cite{Sa1} and \cite{Sa2}).
However, the Wirtm\"uller's proof does not give a construction of generators. 
Moreover, this proof in the case of $D_n$ is less clear than for the case 
of $A_n$. 
An explicit construction of generators of the polynomial ring 
$J_{*,*}^{W}(R)$ (where the root system $R\ne E_8$) is important for 
determining the flat coordinates in the theory of Frobenius varieties
(see \cite{Sa3}, \cite{DZ}, \cite{Sat} and \cite{Ber1}--\cite{Ber2}). 
For example, in \cite{Sat} one can find (without detailed proof) a 
construction of the generators in the $E_6$ case. In 
\cite{Ber1}--\cite{Ber2} the cases of $A_n$, $B_n$ and $G_2$ were considered independently 
from \cite{Wirt}. 
We note that H. Wang proved recently that  $J_{*,*}^{W}(E_8)$ is not 
polynomial (see \cite{Wa}).

As we mentioned above, our interest to generators of $J_{*,*}^{w, W(D_n)}$ 
is explained by existence of the $D_8$-tower of the reflective 
automorphic discriminants starting with the Borcherds--Enriques modular 
form
$$
\Phi_4\in M_4(O^+(U\oplus U(2)\oplus E_8(-2)),\chi_2)=M_4(O^+(U\oplus U\oplus 
D_8(-1)), \chi_2)
$$ 
which is the discriminant of the moduli space of Enriques surfaces (see 
\cite{Bo}, \cite{K} and \cite[\S 5]{Gr18} or Arxiv.1005.3753). 
Due to this relation we are interested in the generators of the 
$O(D_n)$-invariant weak Jacobi modular forms with respect to the root 
lattice $D_n$ where $O(D_n)$ is the full orthogonal group of the lattice 
$D_n$. We note that our construction is quite arithmetic, and it gives us 
a rather interesting system of non-linear differential equations related
to the main arithmetic generators of the bigraded ring of Jacobi forms
(see \S 5).

All reflective automorphic discriminants of the $D_8$-tower
are determined by the last generators
$\varphi_{0,1}\in J_{0,1}^{O(D_n)}$ ($2\le n\le 8$)
of the bigraded ring of $J_{*,*}^{O(D_n)}$, which contains two types of
generators of index one and two. See \S 3.
All  generators of index $2$, namely 
$\varphi_{-2n+2,2}^{D_n}$, 
$\varphi_{-2n+4,2}^{D_n},\, \ldots\,,\varphi_{-6,2}^{D_n}$,
come, in fact, from the sublattice $nA_1=A_1\oplus\ldots\oplus A_1$ 
of $D_n$. We can easily construct them using only Jacobi modular forms
$\varphi_{-2,1}$ and $\varphi_{0,1}$ introduced by Eichler and Zagier (see \cite{EZ}). 

There are four  generators of index $1$, namely
$(\omega_{-n,1}^{D_n})^2$, 
$\varphi_{-4,1}^{D_n}$, $\varphi_{-2,1}^{D_n}$, $\varphi_{0,1}^{D_n}$.
The Jacobi form $\omega_{-n,1}^{D_n}$ has the simplest possible divisor.
It corresponds to the orbit of the miniscule weights of the root system 
$D_n$. So it is proportional to the products of odd Jacobi theta-series.
This form $\omega_{-n,1}^{D_n}$ defines the difference between two graded 
rings
$J_{*,*}^{O(D_n)}$ and $J_{*,*}^{W(D_n)}$ (see more details in \S 4).

We prove in \S 5 that the last three generators $\varphi_{-4,1}^{D_n}$,
$\varphi_{-2,1}^{D_n}$,  $\varphi_{0,1}^{D_n}$ 
satisfy a special system of modular differential equations. Moreover
we found a general modular differential equation for $\varphi_{0,1}^{D_n}$ 
(see (6) in \S 5).

We restrict ourselves in this paper only to the $D_n$ case with 
$n\le 8$ because of existence of the $D_8$-tower of reflective 
automorphic forms. Our arithmetic construction of the generators works 
with small modification for any $n$. We plan to study different variations 
of generators and the corresponding non-linear differential equations for 
arbitrary $n$ in a separate publication. 

\section{Jacobi modular forms and root systems: definitions and examples}

In this paper we use the following definition of weak Jacobi forms 
in many variables. See \cite{Gr94}, \cite{CG} and \cite{Gr18} for more details.

By a lattice we mean a free $\mathbb{Z}$-module equipped  with a
non-degenerate symmetric bilinear form $(\cdot , \cdot)$ with values
in $\mathbb{Z}$. A lattice is even if $(l,l)$ is even for all its elements.

\begin{definition}\label{Jacobi}
Let $L$ be the a positive definite even lattice with the inner product $
(\cdot\, , \cdot)$, a variable $\tau$ be from the upper half-plane $\mathcal{H}$ 
and let $\mathfrak{z}=(z_1,\ldots, z_n)\in L\otimes\mathbb{C}$. Then 
\textit{a weak Jacobi form of weight $k\in \Bbb Z$ and index $m\in \Bbb 
Z$ for the lattice $L$} is a holomorphic function 
$\varphi: 
\mathcal{H}\times(L\otimes\mathbb{C})\to\mathbb{C}$ which satisfies the 
functional equations 
$$
\begin{aligned}
\varphi\left(\frac{a\tau+b}{c\tau+d},\frac{\mathfrak{z}}{c\tau+d}
\right)&=(c\tau+d)^ke^{\pi im\frac{c(\mathfrak{z},\mathfrak{z})}{c\tau
+d}}\varphi(\tau,\mathfrak{z})\text{~~~for any~~} 
\left(\begin{smallmatrix}
a & b\\
c & d
\end{smallmatrix}\right)\in SL_2(\mathbb{Z}),\\
\varphi(\tau, \mathfrak{z}+\lambda\tau+\mu)
&=e^{-2\pi im(\lambda,\,\mathfrak{z})-\pi im(\lambda,\,\lambda)\tau}\varphi(\tau,\mathfrak{z}) \text{~~~for all~~} \lambda, \mu\in L
\end{aligned}
$$
and has a Fourier expansion 
$$
\varphi(\tau,\mathfrak{z})=
\sum_{n\geqslant 0,\, l\in L^{\vee}} a(n, l)q^n \zeta^l,
$$
where $q$ denotes $e^{2\pi i\tau}$ and $\zeta^l$ denotes $e^{2\pi i (\mathfrak{z},\, l)}$ for any $l \in L^{\vee}$, where $L^{\vee}$ is the dual lattice of $L$.
\end{definition}

\begin{remark}
1) If the Fourier expansion of $\varphi(\tau, \mathfrak{z})$ satisfies the stronger condition
$a(n, l)\neq 0 \Rightarrow 2mn\geqslant (l, l)$, 
then such form is called {\it holomorphic} Jacobi form. And if it satisfies even more strong condition
$a(n, l)\neq 0 \Rightarrow 2mn > (l, l)$,
then this form is called {\it cusp} Jacobi form.
We denote the respective finite dimensional spaces by
$$
J^{c}_{k,m}(L)\subset J_{k,m}(L)\subset J^{w}_{k,m}(L).
$$

2) In fact, these spaces depend only on the discriminant group 
of the renormalised lattice $L(m)$. The space of holomorphic Jacobi forms is isomorphic to the holomorphic vector-valued modular forms
defined by Weil representation of  $L(m)$ 
(see \cite[Lemma 2.3]{Gr94}).

3) It is easy to show that there are no non-zero Jacobi forms of 
negative index $m$.  For $m=0$ a weak Jacobi form of weight $k$ is a 
$SL_2(\Bbb Z)$-modular form $f(\tau)$ of weight $k$.
\end{remark}

\begin{definition}
Let $G<O(L)$ be a subgroup of the integral orthogonal group of a positive definite lattice $L$. A weak Jacobi form $\varphi(\tau, \mathfrak{z})$ for the lattice $L$ is called $G$-invariant, if for any $g\in G$
$$
\varphi(\tau, g(\mathfrak{z}))=\varphi(\tau, \mathfrak{z}).
$$
The finite-dimensional space of all $G$-invariant weak Jacobi forms of weight $k$ and index $m$ is denoted by $J_{k,m}^{w,G}(L)$; an analogous notation is used for the holomorphic and cusp forms.)
In the special case of the full integral orthogonal group $O(L)$ of lattice $L$ and its subgroup 
$W(L)$ we have
$$J_{k,m}^{w,O}(L) \subset J_{k,m}^{w,W}(L).$$
\end{definition}

The set of all weak $G$-invariant Jacobi forms has the structure of a bigraded ring
$$
J_{*,*}^{G}(L)=\bigoplus_{k\in \Bbb Z,\,m\in \Bbb Z_{\ge 0}}
J_{k,m}^{w,G}(L)
$$
This ring can be considered as an algebra over the graded ring of all holomorphic modular forms with respect to the full modular group
$$
M_*=\bigoplus_{k\ge 0} M_{2k}(SL_2(\Bbb Z))=\Bbb C[E_4(\tau), E_6(\tau)]
$$
which is the polynomial ring generated by two Eisenstein series
\begin{equation*}\label{Eis}
E_4(\tau)=1+240\sum_{n\ge 1}\sigma_3(n)q^n,\quad 
E_6(\tau)=1-504\sum_{n\ge 1}\sigma_5(n)q^n
\end{equation*}
where $\sigma_k(n)=\sum_{d\vert n}d^k$.

In this paper we consider the case $L=D_n$,  the root lattice of $D_n$ type. Following \cite{Bur}, let us recall the main prorepties of root 
lattices $D_n$.

Let $\varepsilon_1, \ldots, \varepsilon_n$ be the standard orthonormal basis in $\mathbb{Z}^n$. We can choose $\alpha_1=\varepsilon_1-\varepsilon_2,$ $\ldots$, $\alpha_{n-1}=\varepsilon_{n-1}-\varepsilon_n$ and $\alpha_n=\varepsilon_{n-1}+\varepsilon_n$ as the basis of root system $D_n$. These vectors generate even positive definite lattice  
$$
D_n = \{(x_1, \ldots, x_n) \in \mathbb{Z}^n\, |\,
\sum x_i \equiv 0 \mod{2}\},
$$ 
which is a sublattice of index 2 of $\mathbb{Z}^n$.
The root systems of $D_n$-type are well-defined for all $n\geqslant 2$. 
We note that the lattice $D_2$ is isomorphic to the direct sum 
$A_1\oplus A_1$ and $D_3$ is isomorphic to root lattice $A_3$. 
Recall that  $|D_{m}^\vee/D_{m}|=4$ and
$$
D_m^\vee/D_m=
\{0, e_1, \frac{1}2(e_1+\dots+e_{m-1}\pm e_m)\, \bmod D_m\}
$$
is the cyclic group of order $4$ generated by
$\frac{1}2(e_1+\dots+e_m)\bmod D_m$, if $m$ is odd,
and the product of two groups of order $2$, if $m$ is even.
We have the following  matrix
of inner products in the discriminant group of $D_m$
of the non-trivial classes modulo $D_m$
$$
\bigl((\mu_i,\mu_j)\bigr)_{i,j\ne 0}=
\left(\begin{matrix}
\frac{m}4&\frac{1}2&\frac{m-2}4\\
\frac{1}2&1&\frac{1}2\\
\frac{m-2}4&\frac{1}2&\frac{m}4
\end{matrix}\right),\qquad \mu_i\in D_m^\vee/D_m,
$$
where the diagonal elements are taken modulo $2\mathbb{Z}$
and the non-diagonal elements are taken modulo $\mathbb{Z}$.
We note that the discriminant group of $D_m$ depends only on $m\bmod 8$.

The Weyl group $W(D_n)$ acts on the elements of the lattice $D_n$ by 
permutations and changing of the sign of even number of coordinates. For 
all cases, except $n=4$, the Weyl group is the subgroup of index 2 of the 
full integral orthogonal group $O(D_n)$, which acts on coordinates by 
permutations and changing any number of signs. In case $n=4$, the integral 
orthogonal group is larger because there are additional transformations 
that correspond to rotations of Dynkin diagram. We denote by 
$O'(D_4)$ the subgroup of $O(D_4)$ that changes the sign of any number of 
coordinates and permutes them. 

In case of root system $D_n$ the set of coroots $D_n^{\vee}$ coincides 
with $D_n$. So, the highest coroot and the highest root are the same 
and are equal
$$
\alpha^{\vee} = \varepsilon_1 + \varepsilon_2 = \alpha_1+\alpha_{2} 
\quad\mbox{ for } n = 2,
$$
$$
\alpha^{\vee} = \varepsilon_1 + \varepsilon_2 = \alpha_1+\alpha_{2}+
\alpha_{3}\quad \mbox{ for } n = 3,
$$
$$
\alpha^{\vee} = \varepsilon_1 + \varepsilon_2 = \alpha_1+2\alpha_2+
\ldots+2\alpha_{n-2}+\alpha_{n-1}+\alpha_{n}\quad
 \mbox{ for } n\geqslant 4.
$$
The degrees of generators of the ring of $W(D_n)$-invariant polynomials are $2, 4, \ldots, 2n-2$ and  $n$.

Below we introduce some examples of Jacobi forms (see \cite{Gr94}, \cite{CG} for details).

\begin{example}\label{Unimod}
{\it Jacobi theta-series for unimodular lattices.}
For any  positive definite even unimodular lattice $L$ 
one can define the Jacobi theta function
\begin{equation*}\label{ThetaL}
\Theta_L(\tau, \mathfrak{z}) =\sum_{l \in L} q^{(l,l)} \zeta^l.
\end{equation*}
As it can be shown, this function is a $O(L)$-invariant Jacobi form of 
weight $\frac{\text{rk} L}{2}$ and index 1. 
So, for example, for the root lattices $D_8\subset E_{8}$ we have
$$
\Theta_{E_8}(\tau, \mathfrak{z})  
\in J_{4, 1}^{W}(E_8)\subset J_{4, 1}^{W}(D_8).
$$
\end{example}

\begin{example}\label{Discr}{\it The odd Jacobi theta-series and 
the theta-discriminant for $D_n$-lattice.}
The main function in our  constructions of Jacobi forms 
is a classical odd Jacobi theta function 
$\vartheta(\tau, z)=-i\vartheta_{11}(\tau, z)$ (see \cite{M} and
\cite{GN98})
$$
\vartheta(\tau, z) = q^{\frac{1}{8}} \sum_{n\in \mathbb{Z}} (-1)^n 
q^{\frac{n(n+1)}{2}} \zeta^{n+\frac{1}{2}} = 
$$
$$
=-q^{\frac{1}{8}}\zeta^{-\frac{1}{2}}
\prod_{n = 1}^{\infty} (1-q^{n-1} \zeta)(1-q^n\zeta^{-1})(1-q^n).
$$
It is well-known that this function satisfies the following relations:
\begin{gather*}
   \label{eq11}
   \vartheta(\tau,z+x\tau+y)=(-1)^{x+y}\exp(-\pi i(x^2\tau+2x z))
   \vartheta(\tau,z),\quad
   (x,y) \in \mathbb{Z}^2,
   \\
   \label{eq12}
\vartheta\left(\frac{a\tau+b}{c\tau+d},\frac{z}{c\tau +d}\right)
=v_\eta^3(A)(c\tau+d)^{1/2}
   \exp\biggl(\pi i\frac{cz^2}{c\tau+d}\biggr)\vartheta(\tau,z)
\end{gather*}
for any 
$A=\left(\begin{smallmatrix} a&b\\ c&d \end{smallmatrix}\right)\in\operatorname{SL}_2(\mathbb{Z})$,
where $v_\eta$~is the multiplier system of order~$24$ of the Dedekind
eta-function
$$
\eta(\tau)=q^{\frac{1}{24}}\prod_{n=1}^{\infty} (1-q^n).
$$
We note that
\begin{equation*}\label{G2}
\frac{\partial\vartheta(\tau ,z)}{\partial z}\big|_{z=0}
=2\pi i \,\eta(\tau )^3,\quad
\dfrac{\partial^2 \log \vartheta(\tau,z)}{\partial z^2}=
-\wp(\tau,z)+8\pi^2 G_2(\tau),
\end{equation*}
where $\wp(\tau,z)$ is the Weierstrass $\wp$-function and 
$G_2(\tau)=-\frac{1}{24}+\sum_{n\ge 1 }\sigma_1(n)q^n$
is the quasi-modular Eisenstein series.

There is a better way to write all these functional equations.
It is known that 
$$
\vartheta(\tau, z)\in J_{\frac 12, \frac 12}(v_\eta^3\times v_H)
$$
is a holomorphic Jacobi form of half-integral weight and index ($k=1/2$ and 
$m=1/2$) where $v_H$ is a non-trivial binary character of order $2$
of the Heisenberg group (see \cite{GN98}).

Using the Jacobi theta-series we construct the first basic weak Jacobi 
forms of negative weights for for the root systems $A_1$ and $D_n$: 
\begin{equation*}
\begin{aligned}
\varphi_{-2,1}(\tau, z)&=\frac{\vartheta^2(\tau, z)}{\eta^6(\tau)} 
=(\zeta -2 +\zeta^{-1}) +q\cdot(\ldots)
\in J_{-2,1}^{w, W}(A_1),\\
\omega^{D_n}_{-n,1}(\tau, \mathfrak{z})&= 
\frac{\vartheta(\tau, z_1)\cdot \ldots\cdot\vartheta(\tau, z_n)}
{\eta^{3n}(\tau)}\in J_{-n,1}^{w, W}(D_n)\ \  (n\ge 2).
\end{aligned}
\end{equation*}
We note that $\omega^{D_n}_{-n,1}$ is  invariant
with respect to $W(D_n)$ and anti-invariant
with respect to $\sigma\in O(D_n)\setminus W(D_n)$ for $n \neq 4$. 
For $n=4$ it is anti-invariant under the action of the subgroup 
$\sigma\in O'(D_4) \setminus W(D_4)$. 

For $n\equiv 0\mod 8$ we get a holomorphic Jacobi form of so-called
singular (the minimal possible) weight
\begin{equation*}\label{thetaD}
\Theta_{D_{8n}}(\tau, \mathfrak{z})=
\prod_{i=1}^{8n}\vartheta(\tau, z_i)\in J_{4n,1}(D_{8n}).
\end{equation*}
In particular, 
$J_{4n,1}(D_{8})=\Bbb C\langle{\Theta_{E_8}, \Theta_{D_{8}}}\rangle$ (see \cite{CG}).
\end{example}

\begin{example}\label{A1} {\it Hecke operators and construction of generators.} 
We mentioned above that  the Jacobi modular forms of Eichler--Zagier 
are the Jacobi forms for the root lattice $A_1$.
Then $\varphi(\tau, z)$ is $W(A_1)$-invariant if and only if
$\varphi(\tau, z)=\varphi(\tau, -z)$ is even in $z$, i.e.
$\varphi(\tau, z)$ has even weight.

The ring $J_{*,*}^{W}(A_1)=J^{w}_{2*,*}(A_1)$ 
is polynomial in two generators
over the ring of modular forms $M_*$ (see \cite{EZ} and Proposition 2.1 
below). The generators are the weak Jacobi forms $\varphi_{-2,1}$
and $\varphi_{0,1}$. 
There are many ways  to construct the second generator.
For example (see \cite{EZ}), 
one gets it with multiplication by Weierstrass $\wp(\tau,z)$-function:
$$
\varphi_{0,1}(\tau, z)=-\frac{3}{\pi^2}\wp(\tau, z)\varphi_{-2,1}(\tau, z) 
=(\zeta +10 +\zeta^{-1}) +q\cdot(\ldots).
$$
In \cite{GN98}, Hecke operators were used for construction of the same function.We shall apply this method also for $D_{n}$.

In general (see \cite{Gr94}), the action 
of Hecke operator on Jacobi forms of index $1$ for any lattice $L$ 
can be defined by 
\begin{equation*}\label{Hecke}
(\varphi_{k, 1}|{T_{-}(m)}) (\tau, \mathfrak{z})
=m^{-1}\sum_{\substack{ad = 
m \\ b \mod{d}}} a^k \varphi\left(\frac{a\tau+b}{d}, a
\mathfrak{z}\right).
\end{equation*}
This operator multiplies the index of the Jacobi form by $m$ and 
does not decrease the multiplicity of divisor in $z=0$. 
Direct computations give us that
$$
\varphi_{0,1}=
2\frac{\varphi_{-2,1} | T_{-}(2)}{\varphi_{-2,1}(\tau, z)}= 
\frac{1}{2^2}\cdot\frac{\varphi_{-2,1}(2\tau, 2{z})}{\varphi_{-2,1}(\tau, 
{z})} + \frac{\varphi_{-2,1}(\frac{\tau}{2},{z})}{\varphi_{-2,1}(\tau,{z})} 
+\frac{\varphi_{-2,1}(\frac{\tau+1}{2}, {z})}{\varphi_{-2,1}(\tau,{z})}.
$$
\end{example}

\begin{example}\label{Dif} {\it Modular differential operators and 
construction of generators.}
Another construction method makes use of the modular differential operator (see, for example,
\cite{GrJ}). More precisely, 
let us consider the operator $H_k$ that acts on Jacobi form $
\varphi_{k,m}(\tau, \mathfrak{z})$ of weight $k$ and index $m$ for the 
lattice $L$ of rank $n_0$ with the inner product $(\cdot\, , \cdot)$ by the 
formula
$$
H^{(L)}_k(\varphi_{k,m})(\tau, \mathfrak{z}) = $$
$$
=2\pi i \frac{\partial \varphi_{k,m}}{\partial \tau}(\tau, 
\mathfrak{z})+\frac{1}{8\pi^2 m}\left(\frac{\partial}{\partial 
\mathfrak{z}}, \frac{\partial}{\partial \mathfrak{z}}\right)
\varphi_{k,m}(\tau, \mathfrak{z}) + (2k-n_0)G_2(\tau)\varphi_{k,m}(\tau, 
\mathfrak{z})=$$
$$= \sum_{n=0}^{\infty}\sum_{l \in L^{\vee}} \left(n-\frac{1}{2m}(l, l)
\right) a(n, l) q^n \zeta^{l}+  (2k-n_0)G_2(\tau)\varphi_{k,m}(\tau, 
\mathfrak{z}),$$
where $G_2(\tau) = -\frac{1}{24}+q\cdot(\ldots)$ is the quasimodular 
Eisenstein series defined in \eqref{G2}.

In fact, this operator maps weak (holomorhic, cusp) Jacobi form 
of weight $k$ and index $m$ for lattice $L$ to weak (holomorhic, cusp) Jacobi form of weight $k+2$ and index $m$. 
So, in the case $L=A_1$ a direct calculation gives 
$$H_{-2}(\varphi_{-2,1})(\tau, z) = -\frac{1}{24}\varphi_{0,1}(\tau, z)=(\zeta+10+\zeta^{-1})+q\cdot(\ldots).$$
\end{example}

\begin{proposition}\label{wEZ}(see \cite[Theorem 9.3]{EZ})
The bigraded ring $J_{*,*}^{W}(A_1)$ is a polynomial ring
in two generators  over $M_{*}$
$$
J_{*,*}^{W}(A_1)=J_{2*,*}(A_1)=M_{*}[\varphi_{0,1}, \varphi_{-2,1}].
$$
\end{proposition}
\begin{proof}We give a short proof which will be generalised below
to the case of $D_{n}$.

1) We note that $\varphi_{0,1}(\tau,0)=12$ and  $\varphi_{-2,1}(\tau,z)$
has zero of order $2$ in $z=0$.
To show that the generators are algebraically independent over 
$M_{*}$ one has to consider a homogeneous relation of fixed weight and index.  Then one has to consider the specialisation of a polynomial relation for $z=0$.

2) We note that $\varphi_{0,1}(\tau,0)=12$.
For any $\varphi\in J^{w}_{2k,m}$  we have 
$\varphi(\tau,0)=f_{2k}(\tau)\in M_{2k}$. Then 
$$
\psi_{k,m}(\tau,z)=\varphi(\tau,z)-\frac{1}{12^{m}}f_{2k}(\tau)\varphi_{0,1}^{m}(\tau,z)
$$
has zero of even order in $z=0$.  
Therefore  
$\frac{\psi_{k,m}(\tau,z)}{\varphi_{-2,1}(\tau,z)}\in J^w_{k+2,m-1}$
and we can finish the proof by induction in $m$.
\end{proof}

\begin{example}\label{2A1} {\it The case of the orthogonal sum $2A_{1}=A_{1}\oplus A_{1}$.}
The orthogonal group $O(2A_{1})$ is generated by transformations
$$
(z_1,z_2)\mapsto (z_2,z_1),\quad
(z_1,z_2)\mapsto (\pm z_1, \pm z_2).
$$
Therefore any 
$\varphi(\tau, z_1, z_2)\in J_{k,m}^{w, O}(2A_{1})$
is symmetric and even in $(z_1, z_2)$.
In particular, any $O(2A_{1})$-invariant Jacobi form has even weight.

For an orthogonal sum of two lattices one can take the direct symmetric products of Jacobi forms for both lattices:
\begin{equation*}\label{2A_1}
\begin{aligned}
\varphi_{0,1}^{2A_{1}}(\tau, z_{1}, z_{2})&=
\varphi_{0,1}(\tau, z_{1})\varphi_{0,1}(\tau, z_{2}),\\
\varphi_{-2,1}^{2A_{1}}(\tau, z_{1}, z_{2})&=
\varphi_{-2,1}(\tau, z_{1})\varphi_{0,1}(\tau, z_{2})+
\varphi_{0,1}(\tau, z_{1})\varphi_{-2,1}(\tau, z_{2}),\\
\varphi_{-4,1}^{2A_{1}}(\tau, z_{1}, z_{2})&=
\varphi_{-2,1}(\tau, z_{1})\varphi_{-2,1}(\tau, z_{2}).
\end{aligned}
\end{equation*}
\end{example}

\begin{proposition}\label{Prop2A_1}
The bigraded ring $J_{*,*}^{w, O}(2A_1)$ 
is a polynomial ring over $M_{*}$ with three  algebraically independent
generators
$$
J_{*,*}^{w, O}(2A_1)=M_{*}
[\varphi^{2A_{1}}_{0,1}, \varphi^{2A_{1}}_{-2,1}, 
\varphi^{2A_{1}}_{-4,1}].
$$
\end{proposition}
\begin{proof}
We have 
$$
\varphi^{2A_{1}}_{0,1}(\tau,z_{1},0)=12\varphi_{0,1}(\tau,z_{1})
\quad{\rm and} 
\quad
\varphi^{2A_{1}}_{-2,1}(\tau,z_{1},0)=12\varphi_{-2,1}(\tau,z_{1}).
$$
The algebraic independence follows from this fact and 
Proposition \ref{wEZ}. 
For any $\varphi\in J_{2k,m}^{w, O}(2A_{1})$ we have 
$\varphi(\tau,z_{1},0)\in J^{w}_{2k,m}(A_1)$. 
Therefore according to Proposition \ref{wEZ} there exists 
$\psi\in M_{*}(\varphi^{2A_{1}}_{0,1}, \varphi^{2A_{1}}_{-2,1})$ 
of weight $2k$ and index $m$ such that $z_{2}=0$ is zero of even order 
of  $\varphi-\psi$. This Jacobi form is symmetric. Therefore
$(\varphi-\psi)/\varphi^{2A_{1}}_{-4,1}\in J_{2k+4,m-1}^{w, O}(2A_{1})$.
\end{proof}

\section{Constructing of generators for $D_8$}

\subsection{Generators of index 1}

\subsubsection{Weight $0$} 
{\it Hecke operators and two weak Jacobi forms of weight $0$ and index $1$.}
In Example \ref{Discr} we have defined two $W(D_n)$-invariant weak Jacobi 
forms with the simplest divisor $\{z_i=0\}$.
Therefore, in the case of the lattice $D_8$ we can construct two weak Jacobi forms 
of weight 0 using the Hecke operator like in Example \ref{A1}.
The first form is a reflective Jacobi form mentioned in the introduction
(see \cite{Gr18})
\begin{equation*}{\varphi}^{D_8}_{0,1}(\tau, \mathfrak{z}_8)
=-\frac{1}{2}
\frac{\Theta_{D_8}(\tau, \mathfrak{z}_8)| T_{-}(2)}
{\Theta_{D_8}(\tau, \mathfrak{z}_8)}
=8+\sum_{j=1}^8\zeta_j+\sum_{j=1}^8\zeta_j^{-1}+q\cdot(\ldots).
\end{equation*}
The weak Jacobi form of the same type $\omega^{D_8}_{-8,1}$ gives another 
weak Jacobi form of weight $0$
\begin{equation*}\label{psi01}
\psi^{D_8}_{0,1}(\tau, \mathfrak{z}_8)
=\frac{1}{2}\frac{\omega^{D_8}_{-8,1} | T_{-}(2)}{\omega^{D_8}_{-8,1}
(\tau, z)}=512+\sum \zeta_1^{\pm\frac{1}{2}}\ldots\zeta_8^{\pm\frac{1}{2}}
+q\cdot(\ldots).
\end{equation*}
Both Jacobi forms are $O(D_8)$-invariants.
\smallskip

\noindent
\subsubsection{Weight $-4$}
{\it $E_8$-lattice and a weak Jacobi form $\varphi^{D_8}_{-4,1}$ 
of weight $-4$.}
In Example \ref{Unimod} we defined  the Jacobi theta-series 
for an even unimodular lattice.  Note that the lattice 
$D_8$ is a sublattice of unimodular lattice
$E_8 = \langle D_8, \frac{\varepsilon_1+\ldots +\varepsilon_8}{2}\rangle$.
Let us take the unimodular lattice  
$D_{16}^+ 
=\langle D_{16},\frac{\varepsilon_1+\ldots +\varepsilon_{16}}{2}\rangle$. 
We define the Jacobi form of weight $8$ (see \eqref{Eis}) 
without constant term
$$
E_4(\tau)\cdot\Theta_{E_8}(\tau, \mathfrak{z}_8)-\Theta_{D_{16}^+}
\big|_{D_8}(\tau, \mathfrak{z}_8)\in J^{W}_{8, 1}(D_8).
$$
We can  find the $q^1$-term of its Fourier expansion using a description of the roots of these lattices. 
Any root in $E_8$ is equal to $\pm\varepsilon_i\pm\varepsilon_j$ for 
$i\neq j$ or $\frac{\pm \varepsilon_1\pm\ldots\pm\varepsilon_8}{2}$ with 
even number of $+$ and $-$. Hence $q^1$-term of 
$E_4(\tau)\cdot\Theta_{E_8}$
is equal to 
$$
240+\sum_{1\leqslant i< j\leqslant 8} \zeta_i^{\pm1}\zeta_j^{\pm1}+
\sum_{even\,\pm} \zeta_1^{\pm\frac{1}{2}}\ldots\zeta_8^{\pm\frac{1}{2}},
$$
where the second type summands contain even number of $+$ and $-$ in 
exponents. Any root $D_{16}^+$ is equal to 
$\pm \varepsilon_i\pm\varepsilon_j$ for $i\neq j$. Consequently, after 
restriction to $D_8$ the $q^1$-term of Fourier expansion equals
$$
\sum_{1\leqslant i< j\leqslant 8} 
\zeta_i^{\pm1}\zeta_j^{\pm1}+16\sum_{i=1}^8\zeta_i^{\pm 1}+112.
$$
As a result, we get 
\begin{multline*}
\tilde\varphi_{-4,1}(\tau,\mathfrak{z}_8)=\Delta(\tau)^{-1}\biggl(E_4(\tau)
\cdot\vartheta_{E_8}(\tau, \mathfrak{z}_8)-\vartheta_{D_{16}^+}\bigg|_{D_8}
(\tau, \mathfrak{z}_8)\biggr)=\\
128-16\sum_{i=1}^8\zeta_i^{\pm 1}+\sum_{even\,\pm} 
\zeta_1^{\pm\frac{1}{2}}\ldots\zeta_8^{\pm\frac{1}{2}}+q\cdot(\ldots)
\in J_{-4,1}^{w,W}(D_8).
\end{multline*}
By construction, this Jacobi form is $W(D_8)$-invariant. However, it is not 
$O(D_8)$-invariant because of the term 
$\sum_{even\,\pm} \zeta_1^{\pm\frac{1}{2}}\ldots\zeta_8^{\pm\frac{1}{2}}$. 
We can correct this function by the theta-discrimiant
$$
\omega^{D_8}_{-8,1}(\tau,\mathfrak{z}_8)=\sum_{even\,\pm} 
\zeta_1^{\pm\frac{1}{2}}\ldots\zeta_8^{\pm\frac{1}{2}}-\sum_{odd\,\pm} 
\zeta_1^{\pm\frac{1}{2}}\ldots\zeta_8^{\pm\frac{1}{2}} + q\cdot(\ldots),
$$
and obtain $O(D_8)$-invariant form. More precisely,
\begin{multline}
\varphi^{D_8}_{-4,1}(\tau,\mathfrak{z}_8)
=2\tilde{\varphi}_{-4,1}(\tau,\mathfrak{z}_8)
-E_4(\tau)\omega^{D_8}_{-8,1}(\tau,\mathfrak{z}_8)=\\
256-32\sum_{i=1}^8\zeta_i^{\pm 1}
+\sum \zeta_1^{\pm\frac{1}{2}}\ldots\zeta_8^{\pm\frac{1}{2}}+q\cdot(\ldots)
\in J_{-4,1}^{w,O}(D_8),
\end{multline}
where in the last type summand of $q^0$-term the sum is taken for all 
possible combinations of $+$ and $-$.
We note the following relation between the two weak Jacobi forms 
of weight $0$ and index $1$ defined above
$$
\psi^{D_8}_{0,1}-E_4\varphi^{D_8}_{-4,1}=
256+32\sum_{j=1}^8\zeta_j+32\sum_{j=1}^8\zeta_j^{-1}+q
\cdot(\ldots)=32{\varphi}^{D_8}_{0,1}.
$$
\noindent
\subsubsection{Weight $-2$} {\it Modular differential operator and 
$\varphi_{-2,1}^{D_8}$.}
To construct the generator of weight $-2$ and index $1$ 
we apply the modular differential operator from Example 
\ref{Dif}, to $\varphi^{D_8}_{-4,1}$. 
As a result, we get
\begin{equation*}
\varphi^{D_8}_{-2,1}=3\cdot H_{-4}(\varphi_{-4,1}^{D_8})
=512-16\sum_{j=1}^8\zeta_j^{\pm 1}
-\sum \zeta_1^{\pm\frac{1}{2}}\ldots
\zeta_8^{\pm\frac{1}{2}}+q\cdot(\ldots).
\end{equation*}
Applying the modular differential operator $H_{-2}$ to 
$\varphi_{-2,1}^{D_8}$ we obtain another formula for $\psi^{D_8}_{0,1}$
$$
\psi^{D_8}_{0,1}=2\cdot H_{-2}(\varphi_{-2,1}^{D_8})=512+
\sum \zeta_1^{\pm\frac{1}{2}}\ldots\zeta_8^{\pm\frac{1}{2}}+q(\ldots).
$$
\subsection{Generators of index 2 for $D_8$}

We note that by construction (see \S 2) of $D_n$  is a sublattice of index 2 of the lattice $\mathbb{Z}^n$.
Let $L(2)$ denote the lattice $L$ with double inner product, i.e. as sets $L(2)=L$, but for any two elements $l_1, l_2$ their inner product in $L(2)$ is equal to $2(l_1, l_2)$ instead of $(l_1, l_2)$. Then, by 
definition of the lattice $A_1$, $\mathbb{Z}(2)\simeq A_1$ and  
$$
D_n(2)< \mathbb{Z}(2)^{\oplus n}\simeq A_1^{\oplus n}.
$$

Therefore, by analogy with the case of Jacobi forms for $A_1\oplus A_1$ (see Example \ref{2A1}), for any $n$ and $0\le k\le n$ we can define $O(D_n)$-invariant (or $O'(D_4)$-invariant if $n=4$) weak Jacobi forms of weight $-2k$ and index $2$
\begin{multline*}
\varphi_{-2k, 2}^{D_n}(\tau, z_1,\ldots,z_n)= \frac{1}{k!(n-k)!}\sum_{\sigma\in S_n} \varphi_{-2, 1}(\tau, z_{\sigma(1)})\ldots\varphi_{-2, 1}(\tau, z_{\sigma(k)}) \times \\
\times\varphi_{0, 1}(\tau, z_{\sigma(k+1)})\ldots\varphi_{0, 1}(\tau, z_{\sigma(n)}),
\end{multline*}
where the summation is taken over all permutations and  $\varphi_{0,1}$ and $\varphi_{-2,1}$ are Eichler-Zagier's Jacobi forms for $A_1$, which were defined in Example \ref{A1}. Obviously, these forms are also $W(D_n)$-invariant.

Note again that $\varphi_{-2,1}(\tau, 0) = 0$ and $\varphi_{0,1}
(\tau, 0)=12$ (because $\varphi_{0,1}(\tau, 0)=12 + q\cdot(\ldots)$ 
is a modular form of weight zero by definition, but there are no non-constant  modular forms of weight 0). 
For $0\le k\le n-1$ we get
$$
\varphi_{-2k, 2}^{D_n}(\tau, z_1,\ldots,z_{n-1},0)=12\varphi_{-2k, 2}^{D_{n-1}}(\tau, z_1,\ldots,z_{n-1}).
$$ 
Therefore, the constructed Jacobi forms of index 2 form the natural tower 
with respect to the natural embedding $D_2<\dots<D_n$
(see \S 4 below).

\section{The arithmetic $D_8$-tower of weak Jacobi forms}

Now let us move on from the version of Wirthm\"uller's theorem in the case 
of root systems $A_1$ and $A_1\oplus A_1$ to the statement of the main 
results of our work.

\begin{theorem}\label{MainTheorem}
The bigraded ring of all weak Jacobi forms associated with root lattice $D_n$ for $n 
\leqslant 8$ and invariant under the action of the full integral 
orthogonal group $O(D_n)$ in the case $n\neq 4$ and $O'(D_4)$ in the case 
$n=4$ has the structure of the free algebra over the ring of modular forms 
with following generators:
$$J^{w, O}_{*,*}(D_2) = M_*[\varphi^{D_2}_{0, 1}, \varphi^{D_2}_{-2,1}, 
\varphi^{D_2}_{-4,1}],$$
$$J^{w, O}_{*,*}(D_3) = M_*[\varphi^{D_3}_{0, 1}, \varphi^{D_3}_{-2,1}, 
\varphi^{D_3}_{-4,1}, (\omega^{D_3}_{-3, 1})^2], $$
$$J^{w, O'}_{*,*}(D_4) = M_*[\varphi_{0, 1}^{D_4}, \varphi^{D_4}_{-2,1}, 
\varphi^{D_4}_{-4,1},\varphi^{D_4}_{-6,2}, (\omega^{D_4}_{-4, 1})^2],$$
and
$$J^{w, O}_{*,*}(D_n) = M_*[\varphi^{D_n}_{0, 1}, \varphi^{D_n}_{-2,1}, 
\varphi^{D_n}_{-4,1},\varphi^{D_n}_{-6,2}, \ldots, \varphi^{D_n}_{-2n+2, 
2}, (\omega^{D_n}_{-n, 1})^2]$$
for $5 \leqslant n \leqslant 8$.
Moreover, for all $n\neq 4$ all generators, except 
$\omega_{-n,1}^{D_n}$, are invariant under the action of the full integral 
orthogonal group (and the group $O'(D_4)$ in the case $n=4$), 
and there is the following natural tower with respect to restrictions from $D_n$ to $D_{n-1}$ by setting $z_n=0$
$$
\begin{scriptsize}
\begin{array}{ccccccccccc}
\varphi^{D_8}_{0, 1}, & \varphi^{D_8}_{-2,1}, & \varphi^{D_8}_{-4,1}, & 
\varphi^{D_8}_{-6,2}, & \varphi_{-8,2}^{D_8},  & \varphi_{-10,2}^{D_8}, & 
\varphi^{D_8}_{-12, 2}, & \varphi^{D_8}_{-14, 2}, & (\omega^{D_8}_{-8, 
1})^2\\
\downarrow & \downarrow & \downarrow & \downarrow & \downarrow & 
\downarrow & \downarrow & \downarrow & \downarrow \\
\varphi^{D_7}_{0, 1}, & \varphi_{-2,1}^{D_7}, & \varphi_{-4,1}^{D_7}, & 
\varphi_{-6,2}^{D_7}, & \varphi_{-8,2}^{D_7},  & \varphi_{-10,2}^{D_7}, & 
\varphi^{D_7}_{-12, 2}, & (\omega^{D_7}_{-7, 1})^2, & 0 \\
\downarrow & \downarrow & \downarrow & \downarrow & \downarrow & 
\downarrow & \downarrow & \downarrow &  &\\
\varphi^{D_6}_{0, 1}, & \varphi_{-2,1}^{D_6}, & \varphi_{-4,1}^{D_6}, & 
\varphi_{-6,2}^{D_6}, & \varphi_{-8,2}^{D_6}, & \varphi_{-10,2}^{D_6}, & 
(\omega^{D_6}_{-6, 1})^2, & 0 & &\\
\downarrow & \downarrow & \downarrow & \downarrow & \downarrow & 
\downarrow & &\\
\varphi^{D_5}_{0, 1}, & \varphi_{-2,1}^{D_5}, & \varphi_{-4,1}^{D_4}, & 
\varphi_{-6,2}^{D_5}, & \varphi_{-8,2}^{D_5}, & (\omega^{D_5}_{-5, 1})^2, 
& 0 & &\\
\downarrow & \downarrow & \downarrow & \downarrow & \downarrow & 
\downarrow & &  &\\
\varphi^{D_{4}}_{0, 1}, & \varphi_{-2,1}^{D_{4}}, & \varphi_{-4,1}
^{D_{4}}, & \varphi_{-6,2}^{D_{4}}, & (\omega^{D_{4}}_{-4, 1})^2, & 0 & & 
& \\
\downarrow & \downarrow & \downarrow & \downarrow & \downarrow & & &  &\\
\varphi^{D_{3}}_{0, 1}, & \varphi_{-2,1}^{D_{3}}, & \varphi_{-4,1}
^{D_{3}}, & (\omega_{-3,1}^{D_{4}})^2, & 0 &  & & & \\
\downarrow & \downarrow & \downarrow & \downarrow &  &  & &  &\\
\varphi^{D_{2}}_{0, 1}, & \varphi_{-2,1}^{D_{2}}, & \varphi_{-4,1}
^{D_{2}}, & 0 &  &  & & &\\
\end{array}
\end{scriptsize}
$$
where 
$$
\varphi_{-2k,1}^{D_n}\big|_{{z_n}=0}=\varphi_{-2k,1}^{D_{n-1}},\quad
\varphi_{-2k,2}^{D_n}\big|_{{z_n}=0}=12\varphi_{-2k,2}^{D_{n-1}},
\quad $$
$$
\varphi_{-2(n-1),2}^{D_n}=12(\omega^{D_{n-1}}_{-(n-1), 1})^2.
$$
\end{theorem}

%\begin{corollary}\label{WeylCn}
%The set of all weak $W$-invariant Jacobi forms associated with root 
%lattice $C_n$ for $2 \leqslant n \leqslant 8$
%has the structure of the free algebra over the ring of modular forms with 
%generators as in Theorem \ref{MainTheorem}. 
%\end{corollary}

\begin{corollary}\label{WeylDn}
The bigraded ring of all weak $W$-invariant Jacobi forms associated with root 
lattice $D_n$ for $3 \leqslant n \leqslant 8$
has the structure of the free algebra over the ring of modular forms with 
following generators:
$$J^{w, W}_{*,*}(D_3) = M_*[\varphi^{D_3}_{0, 1}, \varphi^{D_3}_{-2,1}, 
\varphi^{D_3}_{-4,1}, \omega^{D_3}_{-3, 1}], $$
$$J^{w, W}_{*,*}(D_4) = M_*[\varphi_{0, 1}^{D_4}, \varphi^{D_4}_{-2,1}, 
\varphi^{D_4}_{-4,1},\varphi^{D_4}_{-6,2}, \omega^{D_4}_{-4, 1}],$$
and
$$J^{w, W}_{*,*}(D_n) = M_*[\varphi^{D_n}_{0, 1}, \varphi^{D_n}_{-2,1}, 
\varphi^{D_n}_{-4,1},\varphi^{D_n}_{-6,2}, \ldots, \varphi^{D_n}_{-2n+2, 
2}, \omega^{D_n}_{-n, 1}]$$
for $5 \leqslant n \leqslant 8$.
Restrictions of these forms to $D_2$ by setting to 0 all coordinates, except 
$z_1$ and $z_2$, give generators of $J^{w, O}_{*,*}(D_2)$ up to 
multiplication by non-zero constants. 
\end{corollary}
\begin{proof}
The proof of this Corollary is the same as the proof of the Theorem \ref{MainTheorem} presented below. We only need to replace the orthogonal groups by Weyl groups.
\end{proof}

%\section{Proofs of the main theorems}
We prove  Theorem \ref{MainTheorem} by induction. 
First of all, we need to consider 
the case $D_2$. In fact, it was made in Proposition 2.2.
Let us note that by definition 
the lattice $D_2$ is equal to the set 
$$\{(x_1, x_2) \in \mathbb{Z} \,|\, x_1+x_2 \equiv 0 \mod{2}\}.$$
The lattice $A_1 \oplus A_1$ is embedded in $D_2$ as the sum of the 
lattices spanned on $(1,1)$ and $(1, -1)$. Actually, $D_2 \simeq A_1 
\oplus A_1$, and in terms described above this isomorphism is given by $
(x_1, x_2) \mapsto (\frac{x_1+x_2}{2}, \frac{x_1-x_2}{2})$. So, in terms 
of of ``weight" coordinates we have isomorphism $D_2 \otimes \mathbb{C}$
($z$-coordinates) and $(A_1 \oplus A_1)\otimes \mathbb{C}$ 
($w$-coordinates) by
$
(z_1, z_2) \leftrightarrow (w_1 + w_2, w_1 - w_2)
$.
We know from Proposition \ref{Prop2A_1},
$$
J_{*,*}^{w, O}(2A_1)=M_{*}
[\varphi^{2A_{1}}_{0,1}, \varphi^{2A_{1}}_{-2,1}, 
\varphi^{2A_{1}}_{-4,1}],
$$
where in $w$-coordinates
$$
\begin{aligned}
\varphi_{-4,1}^{2A_{1}}(\tau, w_{1}, w_{2})&=
\varphi_{-2,1}(\tau, w_{1})\varphi_{-2,1}(\tau, w_{2}),\\
\varphi_{-2,1}^{2A_{1}}(\tau, w_{1}, w_{2})&=
\varphi_{-2,1}(\tau, w_{1})\varphi_{0,1}(\tau, w_{2})+
\varphi_{0,1}(\tau, w_{1})\varphi_{-2,1}(\tau, w_{2}),\\
\varphi_{0,1}^{2A_{1}}(\tau, w_{1}, w_{2})&=
\varphi_{0,1}(\tau, w_{1})\varphi_{0,1}(\tau, w_{2}).
\end{aligned}
$$
Using the inverse transform of coordinates
$
(w_1, w_2) \leftrightarrow (\tfrac{z_1 + z_2}{2}, \tfrac{z_1 - z_2}{2})
$
we obtain
$$\varphi^{D_2}_{-4,1}(\tau, z_{1}, z_{2})=(\zeta_1^{\frac{1}{2}}\zeta_2^{\frac{1}{2}}-2+
\zeta_1^{-\frac{1}{2}}\zeta_2^{-\frac{1}{2}})(\zeta_1^{\frac{1}{2}}
\zeta_2^{-\frac{1}{2}}-2+\zeta_1^{-\frac{1}{2}}\zeta_2^{\frac{1}{2}})
+q(\ldots)=$$
$$=4+\sum_{j=1}^2(\zeta_j+\zeta_j^{-1})+2\sum \zeta_1^{\pm\frac{1}{2}}
\zeta_2^{\pm\frac{1}{2}}+q\cdot(\ldots)=-\frac{1}{32}\varphi_{-4,1}
^{D_8}\big|_{D_2}(\tau, z_{1}, \ldots, z_{8}),$$
$$\varphi^{D_2}_{-2,1}(\tau, z_{1}, z_{2})=(\zeta_1^{\frac{1}{2}}\zeta_2^{\frac{1}{2}}-2+
\zeta_1^{-\frac{1}{2}}\zeta_2^{-\frac{1}{2}})(\zeta_1^{\frac{1}{2}}
\zeta_2^{-\frac{1}{2}}+10+\zeta_1^{-\frac{1}{2}}\zeta_2^{\frac{1}{2}})+
$$
$$
+(\zeta_1^{\frac{1}{2}}\zeta_2^{\frac{1}{2}}+10+\zeta_1^{-\frac{1}{2}}
\zeta_2^{-\frac{1}{2}})(\zeta_1^{\frac{1}{2}}\zeta_2^{-\frac{1}{2}}-2+
\zeta_1^{-\frac{1}{2}}\zeta_2^{\frac{1}{2}})+q\cdot(\ldots)=
$$
$$
=-40+2\sum_{j=1}^2(\zeta_j+\zeta_j^{-1})-8\sum \zeta_1^{\pm\frac{1}{2}}
\zeta_2^{\pm\frac{1}{2}}+q(\ldots)=-\frac{1}{8}\varphi_{-2,1}^{D_8}\big|
_{D_2}(\tau, z_{1}, \ldots, z_{8}),
$$
$$
\hat{\varphi}^{D_2}_{0,1}(\tau, z_{1}, z_{2})=(\zeta_1^{\frac{1}{2}}\zeta_2^{\frac{1}{2}}
+10+\zeta_1^{-\frac{1}{2}}\zeta_2^{-\frac{1}{2}})(\zeta_1^{\frac{1}{2}}
\zeta_2^{-\frac{1}{2}}+10+\zeta_1^{-\frac{1}{2}}\zeta_2^{\frac{1}{2}})
+q(\ldots)=
$$
$$
=100+\sum_{j=1}^2(\zeta_j+\zeta_j^{-1})+10\sum \zeta_1^{\pm\frac{1}{2}}
\zeta_2^{\pm\frac{1}{2}}+q\cdot(\ldots).
$$
The last Jacobi form is not the restriction of the basic form 
$\varphi^{D_8}_{0, 1}$.  If we consider 
\begin{equation}\label{D2gen}
\varphi^{D_2}_{0,1}(\tau, z_{1}, z_{2}) = \hat{\varphi}^{D_2}_{0,1}(\tau, z_{1}, z_{2})+5E_4(\tau)\varphi^{D_2}
_{-4,1}(\tau, z_{1}, z_{2}),
\end{equation}
then
$\varphi^{D_2}_{0,1}=6\varphi_{0,1}^{D_8}\big|_{D_2}(\tau, z_{1}, \ldots, z_{8}).
$
%The replacing $\varphi^{D_2}_{0,1}$ with $\hat{\varphi}^{D_2}
%_{0,1}+5E_4\varphi^{D_2}_{-4,1}$ is linear and non-trivial over the ring 
%of modular forms, so forms $\varphi^{D_2}_{0,1}$, $\varphi^{D_2}_{-2,1}$ 
%and $\varphi^{D_2}_{0,1}$ are also generators.
Therefore, we get 
$$
J^{w, O}_{*,*}(D_2) = M_*[\varphi^{D_2}_{0, 1}, \varphi^{D_2}_{-2,1}, 
\varphi^{D_2}_{-4,1}].
$$

\begin{remark} {\it The bigraded ring of the $W(D_2)$-invariant weak Jacobi forms is not polynomial}. It is a quadratic extension of the  
$O(D_2)$-invariant quadratic ring.
The Jacobi form 
$\varphi_{-2,1}(\tau,z_1)\cdot\varphi_{0,1}(\tau, z_2)-
\varphi_{0,1}(\tau, z_1)\cdot\varphi_{-2,1}(\tau, z_2)$
is not  invariant under the action of  the whole $O(A_1 \oplus A_1)$.
Writing it in $w$-coordinates we get 
$$
\varphi_{-2,1}(\tau, w_1)\cdot\varphi_{0,1}(\tau, w_2)-
\varphi_{0,1}(\tau, w_1)\cdot\varphi_{-2,1}(\tau, w_2)=$$
$$
=12(\zeta_1^{\frac{1}{2}}\zeta_2^{\frac{1}{2}} - \zeta_1^{-\frac{1}{2}}
\zeta_2^{\frac{1}{2}}-\zeta_1^{\frac{1}{2}}\zeta_2^{-\frac{1}{2}}
+\zeta_1^{-\frac{1}{2}}\zeta_2^{-\frac{1}{2}})+q\cdot(\ldots)=12\omega_{-2,1}^{D_2}(\tau, w_1,w_2).
$$
%However,
%$$(\varphi_{-2,1}(\tau, w_1)\cdot\varphi_{0,1}(\tau, w_2)-
%\varphi_{0,1}(\tau, w_1)\cdot\varphi_{-2,1}(\tau, w_2))^2 = $$
%$$=(\varphi_{-2,1}(\tau, w_1)\cdot\varphi_{0,1}(\tau, w_2)+
%\varphi_{0,1}(\tau, w_1)\cdot\varphi_{-2,1}(\tau, w_2))^2-$$
%$$-4\varphi_{-2,1}(\tau, w_1)\cdot\varphi_{0,1}(\tau, 
%w_2)\cdot\varphi_{0,1}(\tau, w_1)\cdot\varphi_{-2,1}(\tau, 
%w_2)=$$
%$$=(\varphi_{-2,1}^{A_1 \oplus A_1}(\tau, w_1, 
%w_2))^2-4\varphi_{0,1}%^{A_1 
%$\oplus A_1}(\tau, w_1, w_2)\varphi_{-4,1}^{A_1 \oplus A_1}(\tau, w_1, 
%w_2).$$
Then we have the following relation:
$$
(\omega_{-2,1}^{D_2})^2=(\varphi_{-2,1}^{D_2})^2-4\varphi_{0,1}^{D_2}\varphi_{-4,1}^{D_2}+5E_4(\varphi_{-4,1}^{D_2})^2.
$$
\end{remark}

%\begin{remark}\label{AlgIndepWeyl}
%The same reasons prove that for $3 \leqslant n \leqslant 8$ constructed 
%$W(D_n)$-invariant forms are also algebraically independent.
%\end{remark}

\begin{lemma}\label{Generability}
For each $3 \leqslant n \leqslant 8$ the $O(D_n)$-invariant (or $O'(D_4)$-invariant in the case of $n=4$) weak Jacobi forms
in Theorem \ref{MainTheorem} generate the algebra $J_{*, *}^{w, O}(D_n)$ 
over the ring of modular forms.
\end{lemma}
\begin{proof}
We have proved the statement of the lemma for $n=2$.
Suppose that the statement holds for $D_n$. Let us 
consider an arbitraty weak Jacobi form  $\Phi_{k, m} \in J_{*, *}^{w, O}(D_{n+1})$. 
Its restriction $\Phi_{k, m} \big|_{z_{n+1}=0}$ is $O(D_n)$ (or $O'(D_4)$-invariant).
%Indeed, 
%the pull-back is at least $W(D_n)$-invariant. 
%Therefore it is invariant under transform $(z_i, z_{n+1}) \mapsto (-z_i, 
%-z_{n+1})$ for all $i$. But the restriction of this condition to $z_{n+1}
%=0$ is equivalent to invariance under transform $z_i \mapsto -z_i$ for 
%all $i$.  
If  $\Phi_{k, m} \big|_{z_{n+1}=0}\not\equiv 0$, then  by induction
$$
\Phi_{k, m} \big|_{z_{n+1}=0} = P_1(\varphi^{D_n}_{0, 1}, \varphi^{D_n}
_{-2,1}, \varphi^{D_n}_{-4,1},\varphi^{D_n}_{-6,2}, \ldots, \varphi^{D_n}
_{-2n+2, 2}, (\omega^{D_n}_{-n, 1})^2),
$$
where $P$ is a polynomial with coefficients from the ring of modular 
forms. We define another  Jacobi form in $J_{k, m}^{w, O(D_{n+1})}$
$$
\Psi_{k, m}=\Phi_{k, m}-P(\varphi^{D_{n+1}}_{0, 1}, 
\varphi^{D_{n+1}}_{-2,1}, 
\varphi^{D_{n+1}}_{-4,1},\frac{1}{12}\varphi^{D_{n+1}}_{-6,2}, \ldots, 
\frac{1}{12}\varphi^{D_{n+1}}_{-2n+2, 2}, \frac{1}{12}\varphi^{D_{n+1}}_{-2n, 2}). 
$$
By construction of the generators we see that  
$\Psi_{k, m}\big|_{z_{n+1}=0}=0$.  
Moreover, $\Psi_{k, m} \big|_{z_{j}=0} \equiv 0$  and the order of zero is 
equal to $2d>0$ because of invariance under the action of $O(D_{n+1})$ (or $O'(D_4)$). 
The second condition from Definition \ref{Jacobi} of Jacobi forms gives 
that divisor 
of  $\Phi_{k, m}$ contains all points $z_j = s+t \tau$ with $s, t \in 
\mathbb{Z}$. Hence the form $\Psi_{k, m}$ is divisible 
by $(\omega_{-(n+1),1}^{D_{n+1}})^d$. So, we can 
consider the holomorphic quotient
$$
\Phi_{k+2d(n+1), m-2d} = 
\frac{\Phi_{k, m}}{(\omega_{-(n+1),1}^{D_{n+1}})^{2d}}.
$$
which is still weak Jacobi form because the Fourier expansion of
$\omega_{-(n+1),1}^{D_{n+1}}$ starts with $q^0$.
If the index $m-2d$ is positive we continue this process.
If $m-2d=0$, then the Jacobi form
$\Phi_{k+2d(n+1), 0}\big|_{z_{n+1}=0}\in J_{*, *}^{w, O}(D_{n})$ 
is a $SL_2(\Bbb Z)$-modular form according our assumption about $D_n$.
\end{proof}

\begin{lemma}\label{AlgIndep}
For each $3 \leqslant n \leqslant 8$ the $O(D_n)$ (or $O'(D_4)$-invariant) weak Jacobi 
forms from  Theorem \ref{MainTheorem} are algebraically independent.
\end{lemma}
\begin{proof} We have proved  this lemma for $d=2$.
Suppose that the statement  holds for all lattices $D_{n_1}$ with 
$n_1 \leqslant n$ but does not hold for $D_{n+1}$. Then there is 
some polynomial relation between the generators for $D_{n+1}$ and  we can 
choose one with minimal degree. 
Consider the restriction of this relation to 
the lattice $D_{n}$ by setting $z_{n+1}=0$. Then the image of the form $
(\omega_{-(n+1),1}^{D_{n+1}})^2$ and only of this form vanishes 
because the constructed functions satisfy the tower condition. If after 
restriction some monomial in the polynomial relation does not vanish, then we obtain the polynomial relation for $D_{n}$ because of 
tower condition. Otherwise, each monomial is divisible by the form $(\omega_{-
(n+1),1}^{D_{n+1}})^2$, and we can get the polynomial relation of lower 
degree. In both cases we get a contradiction.
\end{proof}

\section{System of modular differential equations}
In this section we introduce some applications of previous results.
We show that the most important generators of index one of 
the graded ring $J_{*,*}^{w,O(D_8)}$ satisfy a system of differential 
equations.
Let us consider three $O(D_8)$-invariant generators 
$\varphi^{D_8}_{-4,1}$, $
\varphi^{D_8}_{-2,1}$ and $\varphi^{D_8}_{0,1}$ of index one. As it was 
mentioned above these forms 
satisfy the following differential equations:
\begin{align}\label{difeqs1}
3H_{-4}(\varphi^{D_8}_{-4,1})&=\varphi^{D_8}_{-2,1},
\end{align}
\begin{align}\label{difeqs2}
2H_{-2}(\varphi^{D_8}_{-2,1})-E_4\varphi^{D_8}_{-4,1}&=32\varphi^{D_8}_{0,1}.
\end{align}
We can construct the third equation using the structure of the ring of 
weak Jacobi forms.
Let us consider three $O(D_8)$-invariant of weak Jacobi forms of weight 2:
$E_6\varphi^{D_8}_{-4,1}$, $E_4\varphi_{-2,1}^{D_8}$, 
$H_0(\varphi^{D_8}_{0,1})$.
There is no weak Jacobi forms $\varphi_{2,1}$ of weight $2$ with zero 
$q^0$-term because in this case
$\varphi_{2,1}/\Delta$, where $\Delta$ is the Ramanujan $\Delta$-function of weight $12$, has weight $-10$. Hence, if we find the suitable 
linear combination of these forms with zero $q^0$-term, this combination 
would be equal to zero identically.
Using explicit constructions of all these forms from subsection 
{\bf 3.1} and
$$
H_0(\varphi^{D_8}_{0,1})=16-\sum_{j=1}^8 \zeta_j^{\pm 1} +  q\cdot(\ldots)
$$
we obtain the following equation:
\begin{align}\label{difeqs3}
E_6\varphi^{D_8}_{-4,1}+E_4\varphi_{-2,1}^{D_8}-48H_0(\varphi^{D_8}_{0,1})=0.
\end{align}
Together, equations \eqref{difeqs1}, \eqref{difeqs2} and \eqref{difeqs3} 
give the following  system of modular differential equations:
$$
\left\{
\begin{aligned}
3H_{-4}(\varphi^{D_8}_{-4,1})&=\varphi^{D_8}_{-2,1},\\
2H_{-2}(\varphi^{D_8}_{-2,1})-E_4\varphi^{D_8}_{-4,1}&=32\varphi^{D_8}
_{0,1},\\ E_6\varphi^{D_8}_{-4,1}+E_4\varphi_{-2,1}^{D_8}
&=48H_0(\varphi^{D_8}_{0,1}).
\end{aligned}
\right.
$$

It would be interesting to find an equation of hight degree of the basic 
Jacobi form $\varphi_{0,1}^{D_8}$. This weak Jacobi form provides the 
Borcherds-Enriques modular form, i.e. the automorphic discriminant of the 
moduli space of Enriques surfaces (see \cite{Gr18}). 
To find the general equation we can again use Theorem \ref{MainTheorem}.

Let us consider all forms of weight 8 for the lattice $D_8$ that can be 
obtained from $\varphi^{D_8}_{0,1}$ using differential operators and 
multiplication by modular forms. 
Then, if we construct non-zero form of weight 8 without $q^0$-term, 
we get one of two cases. If $q^1$-term of constructed form is not equal to 
zero, after dividing by $\Delta(\tau)$ we obtain a form proportional to
$\varphi^{D_8}_{-4,1}$. But if $q^1$-term is equal to zero, we get the 
form that is identically equal to zero because of the structure of the 
polynomial ring $J_{*,*}^{w,W}(D_8)$. In any case we obtain non-trivial 
differential equation. 

However, there is a way to find more complicated differential equations. 
For such relations we need to know the first two coefficients in 
Fourier expansion of $\varphi^{D_8}_{0,1}$. 
Using our explicit construction from subsection {\bf 3.1} again we obtain the following part of Fourier expansion 
$$
\varphi^{D_8}_{0,1}(\tau, z_1,\ldots,z_8) 
=8 + \sum_{j=1}^8 \zeta_j^{\pm 1} + $$
$$(128+36\sum_{j=1}^8 \zeta_j^{\pm 1}+8\sum_{j , k = 1}^8 \zeta_j^{\pm 
1}\zeta_k^{\pm 1}-8\sum \zeta_1^{\pm \frac{1}{2}}\ldots\zeta_8^{\pm 
\frac{1}{2}}+\sum_{j, k, l = 1}^8 \zeta_j^{\pm 1}\zeta_k^{\pm 
1}\zeta_l^{\pm 1})\cdot q+\ldots$$
Forms which can be obtained from $\varphi_{0,1}$ are listed in the following table:
\begin{small}
$$
\begin{array}{|c|c|c|c|c|}
\hline
\text{Weight 0} & \text{Weight 2} & \text{Weight 4} & \text{Weight 6} & \text{Weight 8}\\
\hline
 & & & & \\
\varphi^{D_8}_{0,1} & H_0(\varphi^{D_8}_{0,1}) & H_2(H_0(\varphi^{D_8}_{0,1})) & H_4(H_2(H_0(\varphi^{D_8}_{0,1}))) & H_6(H_4(H_2(H_0(\varphi^{D_8}_{0,1}))))\\
 & & & &\\
 &  & E_4 \varphi^{D_8}_{0,1} & H_4(E_4 \varphi^{D_8}_{0,1}) & H_6(H_4(E_4 \varphi^{D_8}_{0,1}))\\
 & & & &\\
 & & &E_4 H_0(\varphi^{D_8}_{0,1}) & H_6(E_4 H_0(\varphi^{D_8}_{0,1}))\\
 & & & &\\
 & & & E_6 \varphi^{D_8}_{0,1} & H_6(E_6 \varphi^{D_8}_{0,1})\\
 & & & &\\
 & & & & E_4 H_2(H_0(\varphi^{D_8}_{0,1}))\\
  & & & &\\
 & & & & E_6 H_0(\varphi^{D_8}_{0,1})\\
  & & & &\\
 & & & & E_4^2 \varphi^{D_8}_{0,1}\\
 & & & &\\
\hline 
\end{array}
$$
\end{small}
In this case an arbitrary form of weight 8 can be written as
$$
F_{8,1}=a_1 H_6\circ H_4 \circ H_2 \circ H_0(\varphi^{D_8}_{0,1})
+a_2H_6\circ H_4(E_4\varphi^{D_8}_{0,1}))+a_3H_6(E_4H_2(\varphi^{D_8}_{0,1}))+
$$
$$
+a_4H_6(E_6\varphi^{D_8}_{0,1})+a_5E_4H_2\circ H_0(\varphi^{D_8}_{0,1})+a_6E_6H_0(\varphi^{D_8}_{0,1})+a_7E_4^2\varphi^{D_8}_{0,1}.
$$
Direct calculation  shows that the  $q^0$-term  of $F_{8,1}$ 
is equal to zero for any $a_1$, $a_2$, $a_3$, $a_4$, $a_5$ if  
$$
\begin{aligned}
a_6&=\frac{1}{27}\cdot (a_1+18a_2+9a_3-27a_4),\\
a_7&=- \frac{1}{162}\cdot (2a_1+36a_2+9a_3-81a_4+9a_5).
\end{aligned}
$$
However, for any $a_1, a_2, a_3, a_4, a_5$ we obtain that the $q^1$-term
of the Fourier expansion of $F_{8,1}$ is also equal to zero.
It follows that 
$$
\frac{F_{8,1}(\tau, z_1,\ldots,z_8)}{\Delta^2(\tau)}\in
J_{-16,1}^{w,O(D_8)}.
$$
But the last space is trivial according to the Theorem \ref{MainTheorem}. Thus we can add the 
third differential equation to the equations (3) and (4)
\begin{multline}\label{DEphi01}
a_1 H_6\circ H_4 \circ H_2 \circ H_0(\varphi^{D_8}_{0,1})
+a_2H_6\circ H_4(E_4\varphi^{D_8}_{0,1}))
+a_3H_6(E_4H_2(\varphi^{D_8}_{0,1}))\\
+a_4H_6(E_6\varphi^{D_8}_{0,1})+a_5E_4H_2\circ H_0(\varphi^{D_8}_{0,1})
+\frac{1}{27}(a_1+18a_2+9a_3-27a_4)E_6H_0(\varphi^{D_8}_{0,1})\\
- \frac{1}{162}(2a_1+36a_2+9a_3-81a_4+9a_5)E_4^2\varphi^{D_8}_{0,1}=0.
\end{multline}

We can compare the modular differential equation \eqref{DEphi01} for 
$\varphi_{0,1}^{D_8}$ with similar equation for 
$\varphi_{0,1}(\tau,z)\in J^{w, W}_{0,1}(A_1)$.
For the last function  we have
\begin{equation}\label{eqA1}
a_1 H_4\circ H_2 \circ H_0(\varphi_{0,1}) + a_2 
E_4H_0(\varphi_{0,1}+a_3H_4(E_4\varphi_{0,1})+a_4E_6\varphi_{0,1}=0
\end{equation}
with 
$$
a_3=\frac{a_1}{50688}-\frac{91a_2}{11} \quad \text{and} \quad 
a_4=\frac{a_1}{76032}-\frac{115a_2}{88}.
$$
We note that $\varphi_{0,1}$ is equal to the elliptic genus of the Enriques 
surfaces (see \cite{Gr99}). Therefore (\ref{eqA1}) above is a general 
differential equation for the elliptic genus of Enriques or $K3$ surface.

We can write a similar equation for the second generator 
$\varphi_{-2,1}(\tau,z)$ for the lattice $A_1$.
As we know (see \cite{GN98}), 
$$H_{-2}(\varphi_{-2,1}^{A_1})(\tau, z) = -\frac{1}{24}\varphi^{A_1}_{0,1}
(\tau, z).$$
If we apply modular differential operator to $\varphi^{A_1}_{0,1}(\tau, z)
$, we obtain
$$H_{0}(\varphi^{A_1}_{0,1})(\tau, z) = -\frac{5}{24}\zeta+\frac{10}{24}-
\frac{5}{24}\zeta^{-1}+q\cdot(\ldots).$$
And therefore we have the following differential equation
$$
H_0\circ H_{-2}(\varphi^{A_1}_{-2,1})-2880E_4\varphi^{A_1}_{-2,1}=0,
$$
because there are no non-zero weak Jacobi forms of index one and weight 
$-10$.

Almost all constructions of this paper can be generalized to the case
of the lattice $D_n$ for any $n>8$.
We are planing to study modular differential equations for the 
generators of the bigraded  ring $J_{*,*}^{w,O(D_n)}$ in the next 
publication.

Dmitry Adler\par
\smallskip
International laboratory of mirror symmetry and automorphic forms
\par
National Research University Higher School of Economics, Moscow
\par
\smallskip
{\tt dmitry.v.adler@gmail.com}
\vskip0.5cm

Valery Gritsenko \par
\smallskip

Laboratoire Paul Painlev\'e, Universit\'e de Lille, France \par
\smallskip
International laboratory of mirror symmetry and automorphic forms
\par
National Research University Higher School of Economics, Moscow
\par
\smallskip
{\tt Valery.Gritsenko@univ-lille.fr}

\end{document}